\documentclass{endm}
\usepackage{endmmacro}

\usepackage{epsfig}
\usepackage{pstricks}
\usepackage{url}
\usepackage{hyperref}
\usepackage{amsmath}
\usepackage{verbatim}

\newcommand{\graph}{\operatorname{G}}

\newcommand{\diam}{{\operatorname{diam}}}

\renewcommand{\epsilon}{\varepsilon}

\newcommand{\simp}{{\textrm{s}}}

\newcommand{\N}{\mathbb N}

\newcommand{\Z}{\mathbb Z}

\begin{document}

	\begin{frontmatter}

		\title
		{Randomized construction of complexes with large diameter}

		\author{Francisco Criado and Andrew Newman}


		\begin{abstract}
We consider the question of the largest possible combinatorial diameter among $(d-1)$-dimensional simplicial complexes on $n$ vertices, denoted $H_s(n, d)$. Using a probabilistic construction we give a new lower bound on $H_s(n, d)$ that is within an $O(d^2)$ factor of the upper bound. This improves on the previously best-known lower bound which was within a factor of $e^{\Theta(d)}$ of the upper bound. We also make a similar improvement in the case of pseudomanifolds.

		\end{abstract}




	\end{frontmatter}


	\section{Introduction}



	Given a pure simplicial complex $C$, one may define the \emph{dual graph} $\graph(C)$ as the graph whose vertices are the top-dimensional faces of $C$ (referred to as \emph{facets}) and whose edges are pairs of facets that intersect at a face of codimension 1 (faces of codimension 1 are referred to as \emph{ridges}). From this definition, one defines the combinatorial diameter of $C$ as the graph diameter of $\graph(C)$. Of course, in general this may be infinite as $\graph(C)$ need not be connected. Thus, here we consider the case where $C$ is \emph{strongly-connected}, that is exactly the case that $\graph(C)$ is connected.

	The most well-known situation in which the combinatorial diameter is considered is in the case that $C$ is a simplicial polytope. In this situation there is the now-disproved \cite{Santos:Hirsch-counter} Hirsch conjecture which stated that if $C$ is a simplicial polytope of dimension $d$ on $n$ vertices, then the diameter of $C$ is at most $n - d$. While this is now known to be false, the so-called polynomial Hirsch conjecture which states that the combintorial diameter of a simplicial polytope on $n$ vertices is bounded by a polynomial in $n$ and $d$ remains open.

	On the other hand, the question can also be considered for other classes of simplicial complexes as a purely combinatorial question. Following the notation of \cite{CriadoSantos}, one defines $H_s(n, d)$ to be the largest combinatorial diameter of any strongly-connected, pure $(d - 1)$-dimensional simplicial complex on $n$ vertices. In \cite{CriadoSantos}, the following is proved:
\begin{theorem}[Theorem 1.2 of \cite{CriadoSantos}]
		For every $d\in \N$ there are infinitely many $n\in\N$ such that:
		\[
		H_\simp(n,d) \ge \frac{n^{d-1}}{(d+2)^{d-1}}-3.
		\]
\end{theorem}

Combined with a trivial upper bound of $\frac{n^{d - 1}}{(d - 1)(d - 1)!}$ (see, for example, Corollary 2.7 of \cite{SantosSurvey}) this shows that $H_s(n, d) = \Theta_d(n^{d - 1})$. This result of \cite{CriadoSantos} gives the previously best-known upper and lower bounds on $H_s(n,d)$, but as $d$ tends to infinity the ratio between the upper bound and the lower bound grows like $e^{\Theta(d)}$. Here we take a step further toward establishing the true value of $H_s(n, d)$ in decreasing this ratio (via a new lower bound) to $\Theta(d^2)$. Specifically, we prove the following lower bound.

\begin{theorem}\label{simplicial}
Fix $d \geq 3$, then $H_s(n, d)$ satisfies
$$(1 - o_n(1)) \frac{1}{4ed^2} \leq \frac{H_s(n, d)d!}{n^{d-1}} \leq \frac{d}{d-1}$$
\end{theorem}

Moreover, our proof takes a different approach than the deterministic construction in \cite{CriadoSantos}. We instead use the probabilistic method in a way similar to the main result of \cite{PrescribedTorsion}.

 We make the comparison between what we do here and what is done in \cite{PrescribedTorsion} more precise when we outline the proof below. At a very basic level, there is a deterministic step and then a probabilistic step. We start with the dimension $(d-1) \geq 2$ and a positive integer $L$ that we want to realize as the combinatorial diameter of our construction. In the deterministic step we build a complex on $\Theta(L)$ vertices that has diameter $L$. In the probabilistic step we take a quotient of the complex that preserves the diameter,  drops the number of vertices to $\Theta(L^{1/(d - 1)})$, and remains a simplicial complex.

This approach also works for the class of pseudomanifolds, which is also considered in \cite{CriadoSantos}. A pseudomanifold (without boundary) is a simplicial complex so that every ridge is contained in exactly two facets.  We denote by $H_{pm}(n, d)$ the maximum diameter of all $(d - 1)$-dimensional, strongly-connected pseudomanifolds on $n$ vertices. A result of \cite{CriadoSantos} shows that $H_{pm}(n, d) = \Theta_d(n^{d - 1})$, but again the ratio between the upper bound and the lower bound is $e^{\Theta(d)}$. Here we improve this to $\Theta(d^3)$:
\begin{theorem}\label{pm}
Fix $d \geq 3$, then $H_{pm}(n, d)$ satisfies
$$(1 - o_n(1)) \frac{1}{4ed^4} \leq \frac{H_{pm}(n, d)d!}{n^{d-1}} \leq \frac{6}{(d + 1)}$$
\end{theorem}

In this case we slightly improve on the upper bound too by using the fact that $\graph(C)$ is $d$-regular when $C$ is a $(d - 1)$-pseudomanifold.

\section{Proof of main result}
The approach will be as in \cite{PrescribedTorsion}, the goal of which is to construct simplicial complexes on few vertices with large torsion groups in homology. The main result of \cite{PrescribedTorsion} shows that for any finite abelian group $G$ and dimension $d$, there exists a simplicial complex on $O_d(\log^{1/d}|G|)$ vertices which realizes $G$ as its top cohomology group. The construction in \cite{PrescribedTorsion} is partially deterministic and partially probabilistic. The deterministic piece constructs a simplical complex $X$ on $O_d(\log|G|)$ vertices that realizes $G$ as the top cohomology group. The probabilistic piece is to use the Lov\'{a}sz Local Lemma to color the vertices of $X$ in a way that allows us to take a quotient of the complex by the coloring to obtain a simplicial complex on the right number of vertices, but without changing the top cohomology group. This technique was further refined in \cite{HomotopyTypes}. Both in \cite{PrescribedTorsion,HomotopyTypes} and here, once we have found a good coloring (and a good initial construction) the quotient is taken according the the following, combinatorial definition.

\begin{definition}(\cite{PrescribedTorsion})
If $X$ is a simplicial complex with a coloring $f$ of $V(X)$ we define the \emph{pattern} of a face to be the multiset of colors on its vertices. If $f$ is a proper coloring, in the sense that no two vertices connected by an edge receive the same color, we define the \emph{pattern complex} $X/f$ to be the simplicial complex on the set of colors of $f$ so that a subset $P$ of the colors of $f$ is a face of $X/f$ if and only if there is a face of $X$ with $P$ as its pattern.
\end{definition}

Our initial construction, that is the deterministic step, for Theorem \ref{simplicial} is quite simple.	For dimension $d - 1$ fixed, we define the \emph{straight  $(d-1)$-corridor on $N$ vertices} to be the pure complex $SC(N, d)$ on $[N]$ where the facets are given by $[1, ..., d], [2, ..., d + 1], [3, ..., d + 2], ..., [N - d + 1, ..., N]$. Clearly the dual graph to $SC(N, d)$ is a path of length $N - d$, so the diameter of $SC(N, d)$ is $N - d$, but it has $N$ vertices. For the probabilistic step we want to color the vertices by a coloring $f$ with $O_d(N^{1/(d - 1)})$ colors so that $SC(N, d)/f$ still has diameter $N -  d$ (moreover it will still have the same dual graph as $SC(N, d)$).

The rule that the coloring $f$ should satisfy is that it should be a proper coloring and it should assign no pair of ridges the same pattern. This rule for coloring vertices is exactly the same as the rule for the result in \cite{PrescribedTorsion}.

The types of properties that are preserved under taking the pattern complex with respect to such a coloring are those which are determined by the top-dimensional boundary matrix (over the integers) of the starting complex. This is slightly stronger than what we need and makes the proof a bit more complicated (to deal with orientation issues), for our purposes it suffices to consider only properties which are determined by the top-dimensional boundary matrix over $\Z/2\Z$.

Recall that the $i$-dimensional boundary matrix $\partial_i$ over $\Z/2\Z$ of a simplicial complex $X$ is a matrix over $\Z/2\Z$ with columns indexed by the $i$-dimensional faces of $X$, rows indexed $i - 1$-dimensional faces of $X$, and the $(\sigma, \tau)$ entry equal to 1 if and only if $\sigma \subseteq \tau$.


	\begin{lemma}\label{lemma1}
		If $\mathcal{P}$ is a property of $(d-1)$-simplicial complexes which is determined by $\partial_{d - 1}$ over $\Z/2\Z$ and $C$ has property $\mathcal{P}$ and $f$ is a proper coloring of $C$ so that every ridge has a unique pattern then $(C, f)$ has property $\mathcal{P}$.
	\end{lemma}
	\begin{proof}
	We show that under the assumptions on $f$ and $C$, both $C$ and $C/f$ have the same $(d - 1)$st boundary matrix over $\Z/2$. We have that $\phi: V(C) \rightarrow V(C/f)$ by sending $v$ to $f(v)$ is a simplicial map and moreover it is injective on the set of ridges and the set of facets (if no two ridges receive the same pattern than certainly no two facets receive the same pattern). Moreover it is clear that for any ridge $\tau$ and any facet $\sigma$ one has $\tau \subseteq \sigma$ if and only if $\phi(\tau) \subseteq \phi(\sigma)$. From this it is immediate that over $\Z/2\Z$, both $C$ and $C/f$ have the same top-dimensional boundary matrix. Hence one will satisfy $\mathcal{P}$ if and only if the other satisfies $\mathcal{P}$
	\end{proof}

Now if $C$ is a $(d - 1)$-dimensional simplicial complex then the dual graph of $C$ is determined by $\partial_{d - 1}$. Indeed the off-diagonal entries of $\partial_{d - 1}^T \partial_{d - 1}$ match those of the adjacency matrix of $\graph(C)$. Moreover, the property that $C$ is a pseudomanifold is determined by $\partial_{d - 1}$ too as it is equivalent to every row of $\partial_{d - 1}$ having exactly 2 non-zero entries.

	For the probabilistic step of our proof we want to properly color the vertices in $SC(N, d)$ using $O_d(N^{1/(d - 1)})$ colors so that no pair of ridges receives the same pattern. Then the pattern complex will still have diameter $N - d$.
	
	As in \cite{PrescribedTorsion} and \cite{HomotopyTypes} the coloring is done in steps.  In particular, there will be two steps to the coloring process. The main reason for this is that the first step allows us to handle pairs of ridges which intersect one another while the second step will handle ridges that do not intersect one another. 
	
	For fixed dimension $d-1$, the two step approach first colors $SC(N, d)$ by a coloring $f$ with $O_d(1)$ colors so that no vertices that are at distance at most two from one another in the 1-skeleton of $SC(N, d)$ receive the same color. This condition implies that there are no pair of intersecting ridges receiving the same pattern. In the second step, one uses the Lov\'{a}sz Local Lemma to color the vertices of $SC(N, d)$ by a coloring $g$ having $O_d(N^{d - 1})$ colors so that no pair of disjoint ridges receive the same pattern. The product of the two colors $(f, g)$ is then the final coloring that we use. 
	
	However, there is another advantage to this two step approach. Namely, we may regard the second coloring as a refinement of the first. Indeed if two disjoint ridges receive different patterns by $f$, then it doesn't matter what happens on them with $g$ as they will still receive different patterns. What this allows for is that if we use \emph{more} colors for $f$ (though still only a number depending on $d$), we can save on the number of colors we need for $g$ in a way that is a net reduction in the number of colors in $(f, g)$.
	
	It is not too hard to check that the 1-skeleton of $SC(N, d)$ has maximum vertex degree $2(d - 1)$, and so by greedy coloring with at most $[2(d - 1)]^2 + 1$ colors we may color the vertices so that no pair at distance two receives the same color. 
This can be refined using a coloring $g$ obtained randomly using the Lov\'{a}sz Local Lemma, but this turns out to give a worse lower bound on $H_s(n, d)$ than the bound in \cite{CriadoSantos}, though it is still $\Theta_d(n^{d - 1})$. 

	Here instead, we are better off taking a random coloring for $f$ which we describe below. This will ultimately allow for fewer colors for $g$ and overall.

	\subsection{The first coloring}

	The purpose of this section is to prove the following:

	\begin{proposition}\label{firstcoloring}
    Let $d$, $c_1 > 6(d - 1)$ and $\epsilon > 0$ fixed. There is a coloring of $SC(N, d)$ with $c_1$ colors such that all pattern classes of ridges have at most size $\frac{(1 + \epsilon)N(d-1)}{\binom{c_1}{d - 1}}$ and no pair of intersecting ridges have the same pattern for $N$ large enough.
	\end{proposition}

  Our proof makes use of the following result on Markov chains:
  \begin{theorem}\label{ergodic}[Theorem 1.10.2 in \cite{Norris}]
    Let $(X_n)_{n\geq 0}$ be an irreducible Markov chain (every state can reach any other with nonzero probability in an arbitrary number of steps) on a finite set of states $S$. Then,
    \[\Pr\left(\frac{V_i(n)}{n} \rightarrow \frac{1}{m_i} \text{ as } n\rightarrow \infty \right) =1 \ \forall i \in S.\]

    Where $m_i$ is the expected number steps to go from state $i$ to itself, and $V_i(n)$ is the \emph{number of visits to $i$ before $n$ steps}.
  \end{theorem}


  \begin{proof}[of Proposition \ref{firstcoloring}]
    We assign a sequence of colors to the vertices $v_1,\dots,v_N$ of $SC(N,d)$ by a greedy randomized approach. At each vertex, we give it a random color of $[c_1]$ that was not used in the last $2(d-1)$ vertices. Since $c_1>6(d-1)>2(d-1)$, this is always possible.

    If two intersecting ridges shared the same subset of colors, in particular it means that two vertices, one from each ridge (but not from the intersection), have the same color. Since the facets of $SC(N,d)$ are consecutive $d$-sequences of vertices, two vertices from two intersecting ridges have to be at most $2(d-1)$ units apart. This proves that this coloring is valid.

    It remains to prove that the patterns are (almost) uniformly distributed. Observe that the randomized coloring procedure defined above defines a Markov Chain on the set:

    \[ S= \left\{ x\in [c_1]^{2(d-1)} : x_i \neq x_j \forall i,j\in [2(d-1)], i\neq j \right\}.\]

    In this Markov chain, we can get from any state $x\in S$ to any other $y$ by choosing first colors not in $x$ or $y$ (which is possible because $c_1>6(d-1)$ and each state has $2(d-1)4$ colors), then chosing colors in $y$. This takes $4(d-1)$ steps independently of the initial and final states. Therefore, the Markov chain is irreducible.

    Hence, by Theorem \ref{ergodic} (which we can apply because it is irreducible), as $N$ grows large enough, the proportion of visits to each state approximates the expected value arbitrarily, with probability $1$. Since the chain is symmetric via permutations of colors, this proportion of visits has to be the same for each state, this is, $\tfrac{1}{{c_1 \choose 2(d-1)} }$.

    Finally, each new state of the Markov chain colors $(d-1)$ ridges in the simplicial complex, the ridges that end precisely at the last vertex that we colored. Since the distribution of states is arbitrarily close to uniform, and the assignation of ridges for each state is symmetric again, the distribution of patterns has to be arbitrarily close to uniform with high probability as well.

    In more explicit terms, since there are ${c_1 \choose (d-1)}$ patterns, for $N$ large enough and any value of $\varepsilon>0$ each pattern appears at most $\frac{(1+\varepsilon)N(d-1)}{{c_1 \choose d-1}}$ times with high probability.

  \end{proof}

  We can generalize this result for faces of any other codimension. This will be relevant for the proof of Theorem \ref{pm}.
  \begin{corollary} \label{coro:firstcoloring}
    Let $d$, $k\in \{1,\dots, d-1\}$, $c_1>6(d-1)$, $\varepsilon>0$ fixed. There is a coloring of $SC(N,d)$ with $c_1$ colors such that all color classes of codimension $k$ faces have at most size $\frac{(1+\varepsilon)N\binom{d-1}{k}}{\binom{c_1}{d-k}}$ and no pair of intersecting $(d-k - 1)$-faces have the same pattern for $N$ large enough.
  \end{corollary}

  \begin{proof}
    We replicate the proof of Proposition \ref{firstcoloring}. Here, the random greedy algorithm picks a color that has not been used in the last $2(d-1)$ as before, which is enough to guarantee that no pair of intersecting $k$ faces have the same pattern.

    The Markov chain remains unchanged, but each state of $S$ corresponds now to $\binom{d-1}{k}$ new $(d-k-1)$-faces that end at the current vertex. This comes from the fact that a $(d-k - 1)$-face is a sequence of $d$ consecutive vertices with $k$ of them removed, but we do not remove the last one (in order to give each face a unique state, and count them once).

    Hence, the distribution on patterns of faces is uniform as well, and arbitrarily close to the expectation. This is, for $N$ large enough, every pattern class will have at most $\frac{(1+\varepsilon)N\binom{d-1}{k}}{\binom{c_1}{d-k}}$ ridges.
  \end{proof}

	\subsection{The second coloring}
	
We are now ready to refine the coloring so that no pair of ridges receive the same pattern. We do this via Proposition \ref{secondstep} which we state in a fairly general way to apply it to the pseudomanifold case later. 
	\begin{proposition} \label{secondstep}
		If $C$ is a $(d - 1)$-simplicial complex and there is a coloring of $C$ on at most $c_1$ colors such that no color class of ridges has size more than $S$, no intersecting ridges receive the same pattern, and for any ridge $\sigma$, there are at most $t$ other ridges which intersect $\sigma$ then there is a refinement of the coloring so that having at most
		$c_1 \lceil \sqrt[d - 1]{e(2tS + 1)} \rceil$ colors and every ridge colored uniquely.
	\end{proposition}
	As is the strategy in \cite{PrescribedTorsion,HomotopyTypes}, we will prove this proposition from the Lov\'{a}sz Local Lemma, which we recall in the symmetric version from \cite{AS} below:
\begin{theorem}[Lov\'{a}sz Local Lemma \cite{EL}]
Let $A_1, A_2, ..., A_n$ be events in an arbitrary probability space. Suppose that each event $A_i$ is mutually independent of all the other events $A_j$ but at most $m$, and that $\Pr[A_i] \leq p$ for all $1 \leq i \leq n$. If $ep(m+1) \leq 1$ then $\Pr[\bigwedge_{i = 1}^n \overline{A_i}] > 0$
\end{theorem}

\begin{proof}[of Proposition \ref{secondstep}]
Let $C$ be the complex described and colored by $f: V(C) \rightarrow [c_1]$ so that every color class of ridges has size at most $S$ and no intersecting ridges receive the same pattern by $f$. We will find a second coloring $g$ so that the final coloring satisfying the conclusions of the statements will be $(f, g)$. The coloring $g$ will be constructed randomly by choosing for each vertex a random color uniformly from a set of $c_2$ colors to be determined later. \\

For $\sigma, \tau$ with $f(\sigma) = f(\tau)$, let $A_{\sigma, \tau}$ denote the event that $\sigma$ and $\tau$ receive the same color by $(f, g)$. Clearly if
$$\Pr\left(\bigwedge_{\substack{(\sigma, \tau) \text{ facets of $C$} \\ \text{s.t. } f(\sigma) = f(\tau)}}\overline{A_{\sigma, \tau}}\right) > 0$$
then there exists a choice for $g$ on $c_2$ colors so that $(f, g)$ colors every ridge uniquely. Indeed, for $\sigma$, $\tau$ with $f(\sigma) \neq f(\tau)$, refining the coloring by $g$ will not cause $\sigma$ and $\tau$ to receive the same pattern under $(f, g)$. We thus apply the Local Lemma to the $A_{\sigma, \tau}$. \\

First, for $\sigma, \tau$ (necessarily disjoint) so that $f(\sigma) = f(\tau)$, we have that there is a bijection $\phi: \sigma \rightarrow \tau$ sending each vertex in $\sigma$ to the unique vertex in $\tau$ that receives the same coloring by $f$. Thus $(f, g)$ gives $\sigma$ and $\tau$ the same pattern if and only if $g(v) = g(\phi(v))$ for all $v \in \sigma$. Thus we have the following bound on $A_{\sigma, \tau}$.
$$\Pr(A_{\sigma, \tau}) \leq \left(\frac{1}{c_2}\right)^{d - 1}$$\\

We now bound, for fixed $(\sigma, \tau)$ the number of pairs $(\sigma', \tau')$ so that $(\sigma', \tau')$ is not independent from $(\sigma, \tau)$ \emph{and so that $f(\sigma') = f(\tau')$}. Clearly if $(\sigma \cup \tau) \cap (\sigma' \cup \tau')  = \emptyset$, then $A_{\sigma, \tau}$ is independent of $A_{\sigma', \tau'}$. Thus, we have that $A_{\sigma, \tau}$ is independent of all but at most $2tS$ other events $A_{\sigma', \tau'}$. Indeed for $(\sigma, \tau)$ fixed, we may assume without loss of generality that $\sigma$ intersects $\sigma'$ or $\tau'$ (for the factor of 2), if $\sigma$ is to intersect $\sigma'$, there are at most $t$ such choices for $\sigma'$, and finally with $\sigma'$ chosen we know $f(\sigma')$ and so there are at most $S$ ridges with the same pattern to choose for $\tau'$. \\

It follows that choosing $c_2$ large enough that
$$e \frac{1}{c_2^{d - 1}} (2tS + 1) \leq 1$$
will imply that with positive probability $A_{\sigma, \tau}$ fails to hold simultaneously for all pairs $\sigma, \tau$ with $f(\sigma) = f(\tau)$. Thus we set $c_2 = \lceil \sqrt[d - 1]{e (2tS + 1)} \rceil$, and complete the proof.

\end{proof}

With Proposition \ref{firstcoloring} and \ref{secondstep} now proved we are ready to prove Theorem \ref{simplicial}.
\begin{proof}[of the lower bound for Theorem \ref{simplicial}]
Let $d \geq 3$ be fixed. The result is asymptotic in $n$, so here we fix $c_1 > 6(d - 1)$ and $\epsilon > 0$, and we will show the constant factor in the lower bound emerges as $c_1 \rightarrow \infty$ and $\epsilon \rightarrow 0$.

Let $N$ be large enough so that there exists a coloring of $SC(N, d)$ with $c_1$ colors so that all color classes of ridges have size at most $S = \frac{(1 + \epsilon)(d - 1)N}{\binom{c_1}{d - 1}}$, and let $f: V(SC(N, d)) \rightarrow [c_1]$ be such a coloring. To apply Proposition \ref{secondstep} we need to upper bound the maximum number of ridges that intersect any given ridge, that is in the notation of the proposition we should find a suitable value for $t$ in the case of $SC(N, d)$.

The facets of $SC(N,d)$ are $d$ consecutive vertices in $[N]$, so the ridges are obtained by taking any $d$ consecutive elements of $[N]$ and removing exactly one element. Thus for any fixed ridge $\sigma$, there are at most $2d^2$ ridges $\tau$ that intersect $\sigma$. Indeed for $\sigma$ given we find a (canonical) facet containing $\sigma$, denoted by $F$. Now $F$ intersects at most $2d$ other facets and each of those contains $d$ ridges. This upper bounds the number of ridges intersecting $\sigma$ by $2d^2$.

With $c_1$, $S$, and $t$ determined we may apply Proposition \ref{secondstep} to say that there is a coloring of $SC(N, d)$ by at most $c_1c_2$ colors where $c_2 = \lceil \sqrt[d - 1]{e(2tS + 1)} \rceil$ so that every ridge has a unique pattern. Let $g: V(SC(N, d)) \rightarrow [c_1c_2]$ be such a coloring. By Lemma \ref{lemma1} the complex $X := SC(N, d)/ g$ has diameter equal to the diameter of $SC(N, d)$. Therefore, the diameter of $X$ is $N - d$, and the number of vertices of $X$ is at most
\begin{eqnarray*}
c_1c_2 &=& c_1 \lceil \sqrt[d - 1]{e(2tS + 1)} \rceil \\
&=& c_1 \left \lceil \sqrt[d - 1]{e\left(2(2d^2) \frac{(1 + \epsilon) (d - 1)N}{\binom{c_1}{d - 1}} + 1\right)} \right \rceil \\
\end{eqnarray*}
Now as $c_1$ tends to infinity $\frac{c_1^{d - 1}}{\binom{c_1}{d - 1}}$ tends to $(d - 1)!$. Thus, given $\delta > 0$ we may set $\epsilon$ small enough and $c_1$ large enough so that for all $N$ large enough we have a complex on at most $(1 + \delta) \sqrt[d - 1]{(d - 1)!4d^3(N - d)e}$ vertices with diameter $N - d$. Letting $n$ be the number of vertices in this complex we have a complex on $n$ vertices with diameter at least
$$\left(\frac{1}{1 + \delta}\right)^{d -1} \frac{n^{d - 1}}{4ed^2d!}$$
As $\delta$ is arbitrary this proves the theorem.
\end{proof}

\section{Pseudomanifold case}
Here we prove Theorem \ref{pm}. In this case we establish a smaller upper bound than in the general case based on the observation that the dual graph of a $(d - 1)$-dimensional pseudomanifold is a $d$-regular graph.

Now for (arbitrary) regular graphs we have the following bound on the diameter, this result is a special case of Theorem 5 of \cite{CaccettaSmyth} and also appears in a different form earlier in \cite{Moon}\\

\begin{theorem}[\cite{Moon,CaccettaSmyth}]\label{GraphBound}
Let $G$ be a connected $d$-regular graph on $n$ vertices then
$$\diam(G) \leq \frac{3n}{d + 1}$$
\end{theorem}

From this it follows that we may upper bound the combinatorial diameter of a pseudomanifold as in Theorem \ref{pm}
\begin{proof}[of the upper bound for Theorem \ref{pm}]
Here we use the notation $f_i$ for a complex to denote the usual $f$-vector entry. Let $C$ be a $(d -1)$-dimensional pseudomanifold on $n$ vertices, then $\graph(C)$ is a $d$-regular graph on $f_{d-1}(C)$ vertices. Thus by Theorem \ref{GraphBound} we have that
$$\diam(\graph(C)) \leq \frac{3 f_{d - 1}(C)}{d + 1}$$
Now, since $C$ is a $(d - 1)$-dimensional  pseudomanifold on $n$ vertices,
$$df_{d - 1}(C) = 2f_{d - 2}(C) \leq 2 \binom{n}{d - 1}$$
Thus,
$$\diam(\graph(C)) \leq \frac{6}{d(d + 1)} \binom{n}{d - 1} \leq \frac{6 n^{d-1}}{(d + 1)!}.$$
\end{proof}

Here we prove Theorem \ref{pm} by the same method as Theorem \ref{simplicial} with a different starting construction. For given $d$, we will start with the boundary of $SC(N, d + 1)$, denoted $\partial SC(N, d + 1)$ and color the vertices as in Lemma \ref{lemma1}. It is clear that $\partial SC(N, d + 1)$ is a pseudomanifold. Indeed it a triangulated sphere as it has the combinatorial type of a stacked polytope. Here the diameter of the starting complex is less obvious than the general case, so we first prove the following lemma to establish a lower bound on the diameter of $\partial SC(N, d + 1)$.

\begin{lemma}\label{pmdiameter}
The diameter of $\partial SC(N, d + 1)$ is at least $\frac{d - 1}{d} N - d $.
\end{lemma}
\begin{proof}
  The facets of $\partial SC(N,d+1)$ are boundary ridges of $SC(N,d+1)$. Ridges of $SC(N, d + 1)$ are sequences of consecutive $d+1$ vertices where one of them is missing. Interior ridges of $SC(N,d+1)$ are sequences of $d$ consecutive vertices of $[N]$, except the first and last such sequences (they are only contained in the first and last facets of $SC(N,d+1)$, and thus, they are boundary ridges). Then, there are two types of boundary ridges of $SC(N,d+1)$: the middle ridges, which are sequences of $d+1$ vertices with one missing (but not the first or the last), plus the two special ridges (the first and the last).

  We denote these two special facets of $\partial SC(N,d+1)$ by $\alpha$ and $\omega$. We denote the remaining (middle) facets by $\sigma_{i,j}=\{i,i+1,i+2,\dots, i+d\} \setminus \{i+j\}$, where $i\in [N-d]$ and $j\in\{1,\dots d-1\}$.

  Let $p$ be a potential function over the middle facets of $\partial SC(N,d+1)$, defined by $p(\sigma_{i,j})= i - \tfrac{j}{d-1}$. We will show that every move from one middle facet to an adjacent middle facet increases this potential at most by $1+\tfrac{1}{d-1}$.

  A move in the dual graph corresponds to removing one vertex from a facet, and adding a new vertex. The set of vertices shared by the initial and final facets form then the ridge connecting the two.

  Since $\partial SC(N,d+1)$ is a pseudomanifold, every facet is adjacent to exactly $d$ facets, corresponding to $d$ choices of a vertex to remove (and the new vertex to add is uniquely determined because in a pseudomanifold every ridge is in exactly two facets). There are three cases:

  \begin{itemize}
    \item We remove the last vertex of $\sigma_{i,j}$. In this case, the adjacent facet is $\sigma_{i-1,j+1}$ if $j\neq d-1$. or to $\sigma_{i-2, 1}$ if $j=d-1$. In both cases, the potential decreased.
    \item We remove a middle vertex $i+j'\in \{i+1,\dots, i+d-1\}$. Then, the adjacent facet is $\sigma_{i,j'}$, and the potential increased at most by $\tfrac{d-1}{d-1}=1$.
    \item We remove the first vertex. If $j=1$, the adjacent facet is $\sigma_{i+2, d-1}$. If $j\neq 1$, the next facet is $\sigma_{i+1, j-1}$. In both cases, the increase in potential is $1+\tfrac{1}{d-1}$.
  \end{itemize}

  Therefore, any step between middle facets of $\partial SC(N,d+1)$ increases the potential at most by $1+\frac{1}{d-1}=\frac{d}{d-1}$. Observe that $\alpha$ is adjacent only to facets with low ($\leq 1$) potential, since its neighborhood is $N(\alpha)= \{ \sigma_{1,j} : j= 1\dots d-1\} \cup \{\sigma_{2,d-1}\}$. Similarly, the neighborhood of the other extremal facet is $N(\omega)=\{\sigma_{N-d,j} : j=2,\dots, d\} \cup \{\sigma_{N-d-1,1}\}$, so $\omega$ is adjacent only to facets with high ($\geq N-d-2$) potential.

  Any path from $\alpha$ to $\omega$ passes only through middle facets, and increases the potential from a starting value at most $1$ to a final value of at least $N-d-2$ potential. And every step between these facets increases the potential at most by $\frac{d}{d-1}$. Hence, the middle part of the path takes at least $\left((N-d-2)-1\right) \frac{d-1}{d}$ steps. And the path from $\alpha$ to $\omega$ is at least two steps longer than that.

  This is, $\alpha$ and $\omega$ are at least $((N-d-2)-1)\frac{d-1}{d}+2= \frac{d-1}{d}N - d +\frac{3}{d} > \frac{d-1}{d}N -d$ steps apart, and the diameter of $\partial SC(N,d+1)$ is at least as large.

\end{proof}

Now we want to color $\partial SC(N, d + 1)$ as in the proof of the general case to apply Lemma \ref{lemma1}.

\begin{proof}[of the lower bound for Theorem \ref{pm}]
As in the proof of Theorem \ref{simplicial}, fix $d \geq 2$ as well as $c_1 > 6(d - 1)$ and $\epsilon > 0$. We let $N$ be large enough so that we may apply Corollary \ref{coro:firstcoloring} to color $SC(N, d + 1)$ by a coloring $f$ with $c_1$ colors so that all color classes of $(d - 2)$-dimensional faces have size at most
$$S = \frac{(1 + \epsilon)N \binom{d}{2}}{\binom{c_1}{d - 1}}.$$
This coloring is then a coloring of $\partial SC(n, d + 1)$ so that no color class of ridges has size more than $S$ and no intersecting ridges receive the same pattern. \\

In order to apply Proposition \ref{secondstep} to $\partial SC(N, d + 1)$, we need to determine an upper bound $t$ for the maximum number of ridges that intersect any given ridge. We may use $t = (d + 1)^3$. Indeed given a ridge $\sigma$, there exists a (canonical) $d$-dimensional face in $SC(N, d + 1)$ which contains it, this $d$-dimensional intersects $2d$ other $d$-dimensional faces, and each of those contains $\binom{d + 1}{2}$ $(d - 2)$-dimensional faces. Thus the number of ridges in the boundary intersecting $\sigma$ is at most $2d \binom{d + 1}{2} \leq (d + 1)^3$. \\

With $c_1$, $S$, and $t$ determined we apply Proposition \ref{secondstep} to color $\partial SC(N, d + 1)$ with at most
$$c_1 \left \lceil \sqrt[d - 1]{e(2(d + 1)^3 \frac{(1 + \epsilon)N \binom{d}{2}}{\binom{c_1}{d - 1}} + 1)} \right \rceil$$
colors so that no pair of ridges receive the same pattern.

Thus, exactly as in the proof of Theorem \ref{simplicial} we have that for any $\delta > 0$, we may set $c_1$ large enough and $\epsilon$ small enough so that for $N$ large enough have a complex on at most
$$(1 + \delta) \sqrt[d - 1]{e(d - 1)!(d + 1)^3 N d^2/2}$$
vertices, which has diameter at least
$$(1 - \delta) \frac{d - 1}{d} N.$$

Letting $n$ denote the number of vertices in this complex we have a complex on $n$ vertices whose diameter is at least
\begin{eqnarray*}
&&(1 - \delta)\left(\frac{1}{1 + \delta}\right)^{d - 1} \frac{2n^{d - 1}}{e(d + 1)^3 [d^3/(d - 1)] (d - 1)!} \\
 &\leq& (1 - \delta) \left( \frac{1}{1 + \delta} \right)^{d - 1} \frac{n^{d - 1}}{4e d^4 d!} \\
\end{eqnarray*}
As $\delta$ is arbitrary, this proves the claim.
\end{proof}

\section*{Acknowledgments}
The authors thank Francisco Santos for several helpful comments on an earlier draft of this article. 

A. Newman gratefully acknowledges funding by the Deutsche Forschungsgemeinschaft (DFG, German Research Foundation) Graduiertenkolleg ``Facets of Complexity" (GRK 2434).

\bibliographystyle{amsplain}
\bibliography{bibliography}

\end{document}